\documentclass{article}
\usepackage{arxiv}
\usepackage{hyperref}
\usepackage[utf8]{inputenc}

\usepackage{setspace}

\usepackage{amssymb,amsmath,amsthm,eucal}

\usepackage[english]{babel}

\usepackage[T1]{fontenc}
\usepackage[utf8]{inputenc}
\numberwithin{equation}{section}

\usepackage{graphicx}
\usepackage{amsfonts,amsmath,amssymb}
\usepackage{booktabs}
\usepackage{tikz}
\usetikzlibrary{shapes,arrows,positioning,calc,decorations.pathreplacing}
\usetikzlibrary{positioning,arrows.meta,shapes,decorations.pathmorphing}

\usepackage{float}
\usepackage{subcaption}
\usepackage{multirow}
\usepackage{algorithm}
\usepackage{algpseudocode}
\usepackage{xcolor}
\usepackage{color}
\def \x {{\bf x}}

\def \v {{\bf v}}

\newtheorem{tr}{Theorem}

\newtheorem{remark}{Remark}
\usepackage{hyperref}

\title{Hierarchical clustering and dimensional reduction for optimal control of large-scale agent-based models}

\author{
 Angela Monti \\
  Istituto per le Applicazioni del Calcolo \lq \lq M. Picone\rq \rq \\
  National Research Council (CNR)\\
  via G. Amendola 122/D, Bari, Italy\\
  \texttt{angela.monti@cnr.it} 
  \And
  Fasma Diele \\
  Istituto per le Applicazioni del Calcolo \lq \lq M. Picone\rq \rq \\
  National Research Council (CNR)\\
  via G. Amendola 122/D, Bari, Italy\\
  \texttt{fasma.diele@cnr.it}
  \And
  Dante Kalise \\
  Department of Mathematics \\ Imperial College London \\
  South Kensington Campus\\ SW7 2AZ London, United Kingdom \\
  \texttt{d.kalise-balza@imperial.ac.uk}
}

\begin{document}

\maketitle
\begin{abstract}
Agent-based models (ABMs) provide a powerful framework to describe complex systems composed of interacting entities, capable of producing emergent collective behaviours such as consensus formation or clustering. However, the increasing dimensionality of these models - in terms of both the number of agents and the size of their state space - poses significant computational challenges, particularly in the context of optimal control. In this work, we propose a scalable control framework for large-scale ABMs based on a twofold model order reduction strategy: agent clustering and projection-based reduction via Proper Orthogonal Decomposition (POD). These techniques are integrated into a feedback loop that enables the design and application of optimal control laws over a reduced-order representation of the system.

To illustrate the effectiveness of the approach, we consider the opinion dynamics model, a prototypical first-order ABM where agents interact through state-dependent influence functions. We show that our method significantly improves control efficiency, even in scenarios where direct control fails due to model complexity. Beyond its methodological contributions, this work also highlights the relevance of opinion dynamics models in environmental contexts - for example, modeling the diffusion of pro-environmental attitudes or decision-making processes in sustainable policy adoption - where controlling consensus formation plays a crucial role.
\end{abstract}

\section{Introduction}
\label{sec:intro}
Systems of agents that act and interact within a network can be effectively described using agent-based models (ABMs). These models account for the rules governing each agent (viewed as an individual entity) and their interactions. Although the behaviour of individual agents is typically governed by simple, easily explainable laws, their interactions can give rise to complex collective dynamics. The primary objective of ABMs is to reproduce and predict the emergence of such complex phenomena. In the absence of control, agents interact through simple pairwise rules - such as attraction, repulsion, and alignment - which often lead to self-organising behaviours such as flocking or pattern formation. Such systems arise in mathematical biology, ecology, and social dynamics. These include swarming behaviour, crowd dynamics, opinion formation, synchronization, and many others (see, e.g., 
\cite{Ameden2009,bonnans2006computation,Caplat2008,choi2017emergent,Toscani2006} and references therein).

However, the sensitivity of these emergent behaviours to the agents’ initial configuration underscores the need for robust control strategies. Effective control design enables interventions that guide the system toward desired configurations. Within the framework of first-order ABMs, here we focus on analyzing and promoting consensus formation among agents, under externally imposed control inputs.

The complexity of such systems is typically quantified by their dynamic order - that is, the number of state variables in the system’s state-space representation. This intrinsic complexity already poses significant challenges for simulation and analysis. When control strategies are introduced, the overall system becomes even more demanding, as the control design must be embedded within the already high-dimensional dynamics. As a result, model reduction becomes not only useful, but essential, to ensure that the control problem remains computationally tractable.

In recent years, several approaches have emerged to address these issues by leveraging kinetic theory and mean-field models, which enable the design of scalable control strategies over ensemble dynamics. A kinetic framework for optimal feedback control of particle systems was introduced in \cite{Albi2017}, where the continuous distribution of agents facilitates analytical tractability of the control problem. This paradigm has been further extended through the integration of supervised learning techniques into kinetic control settings, allowing for binary feedback representations in high-dimensional agent populations \cite{Albi2022}.

Complementary to these methodologies, moment-based predictive control schemes have been proposed to design feedback laws based on macroscopic observables, thus circumventing the need to track the full state evolution of individual agents \cite{Albi2022_moment}. On the theoretical side, recent advances in controllability analysis provide conditions under which large-scale particle systems can be steered toward desired configurations via local and global control strategies \cite{Carrillo2022}.

Moreover, while kinetic theory and mean-field models offer tractable formulations of the control problem by leveraging ensemble dynamics, these approaches inherently rely on continuum assumptions. In parallel, a complementary line of research has focused on model reduction techniques. These approaches operates directly on the ABM, with the aim of reducing complexity while preserving essential dynamical features. In \cite{Monshizadeh2013,Monshizadeh2014}, the authors proposed projection-based model reduction schemes for linear multi-agent systems, relying on graph partitioning techniques. More recently, the clustering-by-projection approach has been extended to a broader class of nonlinear network systems \cite{Benner2021}.

In this work, we address the problem of reducing both the number of agents and the dimensionality of their states, with the ultimate goal of enabling efficient control of large-scale ABMs. To this end, we first develop and analyze two distinct reduction strategies, each addressing a different aspect of model complexity, before integrating them into a unified framework.

We begin by reducing the number of agents through a DBSCAN clustering algorithm, which groups agents exhibiting similar dynamical behaviour. The resulting ABM describes the evolution of the center of mass of each cluster, providing a compact representation of the collective dynamics.

In parallel, we consider projection-based model order reduction (MOR) techniques - specifically Proper Orthogonal Decomposition (POD) \cite{Sirovich} - to reduce the state dimension of the agents. By constructing an appropriate snapshot matrix, we define a projection subspace onto which we project the ODE system of each agent or cluster. Using Galerkin projection, we obtain a reduced order model (ROM) that preserves the structure of the original system while significantly lowering computational complexity.

These two complementary approaches are then integrated into a novel reduction-and-control framework, which allows scalable and effective control of complex agent-based systems.

As a guiding example, we focus on a classical model of opinion dynamics, where agents iteratively adjust their opinions through pairwise interactions. Opinion dynamics models have received significant attention in recent years not only for their theoretical insights but also for their relevance in practical domains, such as the diffusion of sustainable behaviours and public acceptance of environmental policies. Understanding and controlling the mechanisms behind consensus formation in such models is of growing importance in environmental applications, where fostering alignment on pro-environmental values or climate action can play a critical role in societal transformation \cite{shin2025climate, giardini2021opinion, sun2013framework}.

The remainder of this paper is organized as follows. In Section \ref{sec:optimal_control}, we introduce the class of agent-based models under consideration, along with the formulation of the associated optimal control problem and the numerical methods based on Pontryagin's first-order necessary conditions. Section \ref{sec:red_agents} is devoted to the reduction of the number of agents via clustering, and the construction of a cluster-level dynamical system. In Section \ref{sec:red_dim}, we apply projection-based model order reduction techniques to reduce the state dimension of individual agents or clusters. In Section \ref{sec:framework}, we present the unified framework that combines the two reduction strategies and demonstrate its application to the controlled system. Finally,  Section \ref{sec:conclusion} concludes the paper with a discussion and future research directions.

\section{First-order ABM model and optimal control problem formulation}
\label{sec:optimal_control}
\subsection{The opinion dynamics}
To illustrate the general framework for implementing model order reduction techniques in the optimal control of large-scale ABMs, we consider the concrete example of first-order ABMs describing opinion dynamics \cite{motsch2014heterophilious}.
In these models, $N$ agents interact
with each other according to the first-order system
\begin{equation}
\label{abm}
\dot{x}_i = \sum_{j=1}^{N} a_{i,j} \, (x_j - x_i), \quad a_{i,j}:=\frac{\phi(\|x_i - x_j\|)}{N}, \quad x_i(0)= x_i^{(0)}.
\end{equation}
Here, the function \( \phi \), with \( 0 < \phi < 1 \), models the interaction between agents, each having a vector of opinions represented by the state \( x_i \in \mathbb{R}^d \). The interaction depends only on the relative distance between agents' opinions, measured using the Euclidean norm. Notice that, for simplicity, in model~(\ref{abm}), the adjacency matrix with entries \( a_{i,j} \) is taken symmetric. However, alternative forms of interaction can be considered, such as in the well-known Hegselmann-Krause model \cite{dietrich2016transient}, where the adjacency matrix, defined as
\[
a_{i,j} := \frac{\phi(\|x_i - x_j\|)}{\sum_k \phi(\|x_i - x_k\|)},
\]
is not symmetric.

Our main interest lies in the emergence of consensus, where all agents asymptotically converge to a common opinion $\bar{x}(t)$. Formally, a solution of \eqref{abm} achieves consensus if
\begin{equation}
\label{consensus}
\lim_{t \to +\infty} \|x_i(t) - \bar{x}(t)\| = 0, \quad \forall i = 1, \ldots, N,
\end{equation}
where
\begin{equation}
\label{consensus_state}
\bar{x}(t) = \frac{1}{N} \sum_{i=1}^{N} x_i(t)
\end{equation}
is the mean state of the system.
To study the emergence of consensus, it is useful to define the consensus parameter  
\begin{equation}
 \label{consensus_parameter}
 X(t) = \frac{1}{N^2} \sum_{i=1}^{N} \|x_i(t) - \bar{x}(t)\|^2,
\end{equation}
so a solution to \eqref{abm} tends to consensus if and only if $X(t) \to 0$ as $t \to \infty$.

The long-term emergence of consensus is determined  by the specific properties of the interaction kernel $\phi$ and by the initial configuration $x_i^{(0)}$ of the agents. Results have shown that both the support and monotonicity of $\phi$ play a crucial role in the system's convergence properties. For instance, in models with \emph{global interactions} - where agents influence each other uniformly - consensus can emerge unconditionally for a broad class of initial states, provided the interaction strength remains sufficiently high. In contrast, when $\phi$ is \emph{compactly supported}, the evolution of the system becomes highly sensitive not only to the initial spatial distribution, but also to the \emph{connectivity structure} induced by $\phi$ over time. In these settings, the emergence of clusters is possible, and consensus may fail unless certain uniform connectivity conditions are satisfied. This interplay between initial conditions and the structure of $\phi$ is therefore central in understanding both the asymptotic behaviour and the design of control strategies for ABMs \cite{motsch2014heterophilious}.

Control strategies are employed not only to enforce consensus, but also to address several theoretical and practical challenges. First, control may be required to steer the system toward consensus in scenarios where the natural dynamics lead to cluster formation. Second, even when consensus is guaranteed asymptotically, control can be used to \emph{accelerate} convergence, which is crucial in time-sensitive applications. Third, control becomes essential to counteract numerical artifacts - such as spurious clustering due to discretization - which may otherwise prevent convergence even in systems that theoretically tend to consensus. For all these reasons, control design plays a central role in enabling robust and efficient collective behaviour in large-scale agent-based systems.

These considerations motivate the use of an optimal control framework, where the objective is to steer the system toward consensus while minimizing control effort. This approach provides a principled methodology for balancing coordination performance with energy efficiency, as detailed in the following section.

\noindent
\subsection{The optimal control problem} 
In this Section, we consider the problem of designing a centralized forcing term \( \mathbf{u}(t) = (u_1(t), \ldots, u_N(t)) \), acting on all agents, in order to steer the system toward consensus. The aim is to minimize both the dispersion of the agents' positions from their mean and the control effort.

Formally, for \( T > 0 \), we define the set of admissible controls for the entire population as
\[
\mathcal{U}^N : \mathbb{R}_0^+ \rightarrow [L^\infty([0,T]; \mathbb{R}^d)]^N,
\]
where each control signal \( u_i(\cdot) \in \mathcal{U} \) takes values in 
$\mathbb{R}^d$. The optimal control problem is then formulated as
\begin{equation}
\min_{\mathbf{u}(\cdot)\in \mathcal{U}^N} \mathcal{J}(\mathbf{u}(\cdot); \mathbf{x}_0) ,
\quad \mathbf{x}_0=( x_1^{(0)}, \dots, x_N^{(0)}),
\label{eq:cost_functional_first_order}
\end{equation}
with the cost functional defined by
\begin{equation}
    \mathcal{J}(\mathbf{u}(\cdot); \mathbf{x}_0) := \int_0^T \ell(\mathbf{x}(t), \mathbf{u}(t)) \, dt,
\end{equation}
where
\begin{equation}
\ell(\mathbf{x}(t), \mathbf{u}(t)) := \frac{1}{N} \sum_{i=1}^N \left( \|x_i(t) - \bar{x}(t)   \|^2 + \gamma \|u_i(t)\|^2 \right),
\label{eq:running_cost_first_order}
\end{equation} 
$\bar{x}(t)$ is defined in \eqref{consensus_state} and $\gamma >0$ is a control penalization parameter. 

The minimum of the functional cost is subject to the controlled dynamics:
\begin{equation}
    \label{abm_control}
    \dot{x}_i = \frac{1}{N} \sum_{j=1}^{N} \phi(\|x_i-x_j\|) \left(x_j - x_i\right) + u_i, \quad x_i(0)= x_i^{(0)}, \quad i=1,\dots, N.
\end{equation}
In this way, the control action is designed to penalize both the deviation of the agents' positions from the target consensus state $\bar{x}$ and the control effort required to steer the system.

\subsection{First order necessary conditions for optimality}
\label{sec:opt_nec_cond}
\noindent The existence of a minimizer \(\mathbf{u}^*\) for problem \eqref{eq:cost_functional_first_order} follows from the smoothness and convexity of the system dynamics and the functional cost. To characterize such an optimal control, we apply the Pontryagin's Minimum Principle \cite{pontryagin2018mathematical}, which provides first-order necessary conditions.
Let $p_i \in \mathbb{R}^d$ denote the adjoint variable associated with the state $x_i \in \mathbb{R}^d$, for $i = 1, \dots, N$ and we define the vectors:
\[
\mathbf{x}(t) = (x_1(t), \dots, x_N(t)) \in \mathbb{R}^{dN}, \quad
\mathbf{p}(t) = (p_1(t), \dots, p_N(t)) \in \mathbb{R}^{dN}.
\]
Then, the corresponding Hamiltonian function is given by
\begin{equation}\label{eq:hamiltonian}
    \mathcal{H}(\mathbf{x}, \mathbf{u}, \mathbf{p}) = \frac{1}{N} \sum_{i=1}^{N} \left( \|x_i - \bar{x}\|^2 + \gamma \|u_i\|^2 \right)
    + \sum_{i=1}^{N} p_i^T \left( \frac{1}{N} \sum_{j=1}^{N} \phi(\|x_i - x_j\|)(x_j - x_i) + u_i \right).
\end{equation}

From the structure of the Hamiltonian, we recover the state dynamics via
\begin{equation}
    \label{state}
\dot{x}_i = \nabla_{p_i} \mathcal{H}(\mathbf{x}, \mathbf{u}, \mathbf{p}),
\end{equation}
   and obtain the adjoint equations as
\begin{equation}
    \label{costate}
\dot{p}_i = -\nabla_{x_i} \mathcal{H}(\mathbf{x}, \mathbf{u}, \mathbf{p}),
\end{equation}
with transversality conditions
$$ p_i(T) = 0,$$
for all $i = 1, \dots, N$.
Under the hypothesis of $\phi$ sufficiently regular, the explicit form of the adjoint gradient reads
\begin{equation}
    \label{com_adjoint}
    \nabla_{x_i} \mathcal{H}(\mathbf{x}, \mathbf{u}, \mathbf{p}) =
    \frac{2}{N} (x_i - \bar{x}) +
    \frac{1}{N} \sum_{j \neq i} \left(
        \frac{\phi'(\|x_i - x_j\|)}{\|x_i - x_j\|} \langle p_j - p_i, x_j - x_i \rangle (x_j - x_i)
        + \phi(\|x_i - x_j\|)(p_j - p_i)
    \right).
\end{equation}

The optimality condition with respect to the control yields
\[
\nabla_{u_i} \mathcal{H}(\mathbf{x}, \mathbf{u}, \mathbf{p}) = 0,
\]
which leads to the closed-form control law
\begin{equation}
    \label{com_optimality}
    u_i =  -\frac{N}{2\gamma} \, p_i,
\end{equation}
for all $i = 1, \dots, N$.

\subsection{Fixed-horizon iterative optimal control strategy}

In many practical scenarios, the final time horizon $T > 0$ required to achieve consensus may be large and is not known a priori. This presents a challenge for solving the associated Hamiltonian boundary value problem \eqref{state}-\eqref{costate} arising from the necessary optimality conditions, particularly when using classical forward-backward sweep methods. These methods are prone to instability or divergence if the initial guess for the control is not sufficiently close to the optimal solution.

To mitigate this issue, regularization techniques have been proposed in the literature (see, e.g., \cite{li2018maximum}), demonstrating global convergence in the continuous case. Additionally, symplectic integrators have been employed for their well-established advantages in preserving the geometric structure of Hamiltonian systems.  The theory of geometric numerical integration and backward error analysis supports the superior performance of symplectic methods, as they approximate the trajectories of modified Hamiltonian systems with high accuracy. Works such as \cite{hager2000runge, bonnans2006computation, ragni2010steady, diele2011exponential}, have shown that symplectic Runge-Kutta methods lead to consistent results across both direct and indirect methods in optimal control.

Nevertheless, even regularized methods based on symplectic algorithms may struggle in the presence of large or uncertain horizons. To address this issue, we propose an alternative approach based on fixed-horizon iterative control. Instead of solving a single optimal control problem over a long or unknown horizon, we iteratively solve local problems on short intervals of fixed length $h$, denoted by $[t_m, t_m + h]$. At each iteration $m$, the control is computed over the entire interval, but only the value at the initial time $t_m$ is applied to evolve the system. This results in a piecewise constant control strategy that is updated at each step.
Although this approach is suboptimal with respect to the full time frame optimization, in practice it produces very satisfactory results.

Let \( \mathbf{x}_0^{(0)}:=\mathbf{x}_0 \) denote the initial configuration of the system. At each iteration \( m = 1, \dots, M \), the following steps are performed:

\begin{enumerate}
    \item Solve the optimal control problem on \([t_m, t_m + h]\):
    \[
    \min_{\mathbf{u}^{(m)} \in \mathcal{U}^N} \int_{t_m}^{t_m+h} \ell\bigl(\mathbf{x}^{(m)}(t), \mathbf{u}^{(m)}(t)\bigr) \, dt,
    \]
    subject to the dynamics
    \[
    \dot{x}_i^{(m)} = \frac{1}{N} \sum_{j=1}^N \phi(\|x_i^{(m)} - x_j^{(m)}\|)\,(x_j^{(m)} - x_i^{(m)}) + u_i^{(m)},
    \quad i = 1, \dots, N,
    \]
    with initial condition \(  \mathbf{x}_0^{(m-1)} \).

    \item 
    Apply only the control value $u^{(m)}(t_m)$ to evolve the system for one time step $h_t \ll h$.

    \item Update the initial state:
    \[
    \mathbf{x}_0^{(m+1)} := \mathbf{x}^{(m+1)}(t_{m+1})
    \]
with $t_{m+1} = t_m + h_t$. 
    \item If the consensus condition is satisfied:
     $$X(t_{m+1}) < \varepsilon$$
    with \( \varepsilon > 0 \) a prescribed tolerance (e.g. machine precision), the iteration stops.
\end{enumerate}

This iterative strategy allows for adaptive control design without prior knowledge of the total time required to achieve consensus. It also improves numerical stability by avoiding the need to solve stiff or unstable boundary value problems on long intervals.

\subsection{Low-dimensional example: generalized Hegselmann-Krause
influence functions on infinite support.}

To demonstrate the effectiveness of the proposed optimal control framework, we use the smoothed generalized Hegselmann–Krause influence functions on infinite support \cite{dietrich2016transient}. We choose to consider such functions using a sigmoid according to the formula
$$
\phi_{GHK}(x) = \frac{1 - \text{sig}(\alpha(x - 1))}{1 - \text{sig}(-\alpha)}, \quad \text{sig}(y) = \frac{1}{1 + e^{-y}}$$
where \( \alpha \) is a parameter to control the smoothness of the curve. 
As $\alpha$ increases, the smoothed interaction function increasingly resembles the original Hegselmann-Krause step function with a communication range of $1$, as introduced in \cite{rainer2002opinion}. However, this original formulation cannot be directly incorporated into the proposed optimal control framework due to its lack of sufficient regularity.
Note that lower values of $\alpha$ enlarge the support of the interaction function, thereby promoting faster and more global convergence towards consensus. In contrast, larger values of $\alpha$ make the function increasingly behave like one with compact support, which tends to promote the formation of persistent opinion clusters over long time horizons. 

To illustrate the influence of the parameter $\alpha$ on the opinion dynamics, we consider a system of $N = 10$ agents evolving in $\mathbb{R}^2$ according to the interaction dynamics \eqref{abm}, and we simulate the uncontrolled system using the smoothed interaction function $\phi_{GHK}(x)$ for the four representative values $\alpha = 0.1$, $\alpha = 1.6$, $\alpha = 5$, and $\alpha = 300$. The left panel of Figure \ref{fig:alpha} shows the evolution of the agents' opinions under the smoothed generalized Hegselmann–Krause dynamics for different values of the parameter $\alpha$. 

For $\alpha = 0.1$, the interaction function decays slowly and remains significantly positive even at large distances. As a result, agents influence each other almost globally, leading to a rapid convergence to consensus. For $\alpha = 1.6$, the function decays more sharply, resembling the exponential kernel $\phi(x) = e^{-x}$. In this case, opinions still converge to a single cluster, but over a much longer time horizon, reflecting the reduced but still significant range of interaction. In the case $\alpha = 5$, the influence becomes more localized. Initially, multiple clusters form, but some of them eventually merge. The result is the formation of a few larger opinion groups rather than a single consensus. This intermediate regime highlights the transition between global consensus and persistent fragmentation.

Finally, for $\alpha = 300$, the smoothed interaction function closely approximates the original discontinuous Hegselmann-Krause model. Influence is restricted to very close neighbours, and as a result, multiple distinct clusters emerge and persist indefinitely. The system lacks global consensus, and the initial fragmentation of opinions is preserved over long time scales.

\begin{figure}[hbtp]
    \centering
    \includegraphics[width=0.4\textwidth]{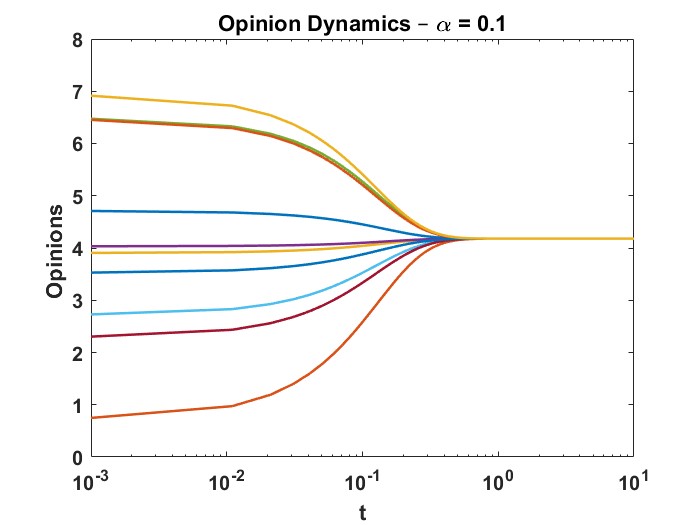}
    \includegraphics[width=0.4\textwidth]{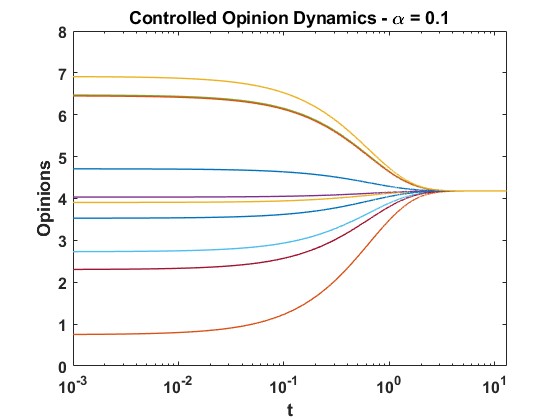}
 \includegraphics[width=0.4\textwidth]{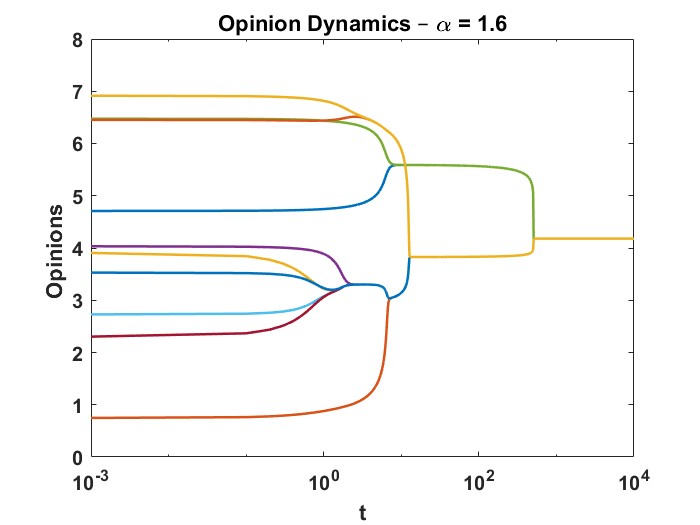}
 \includegraphics[width=0.4\textwidth]{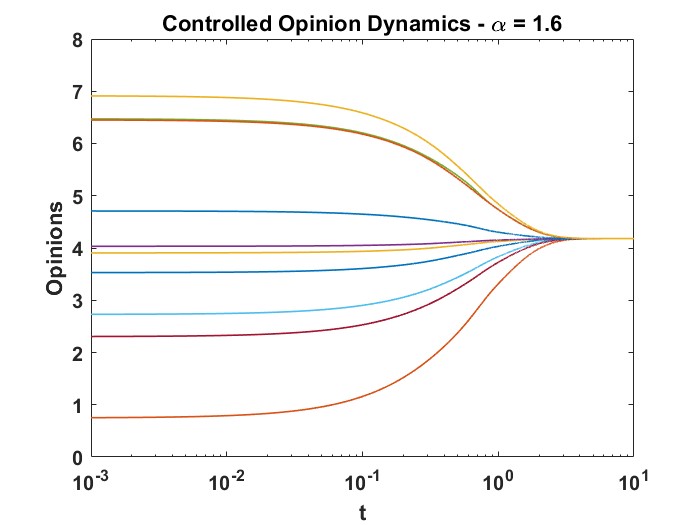}
 \includegraphics[width=0.4\textwidth]{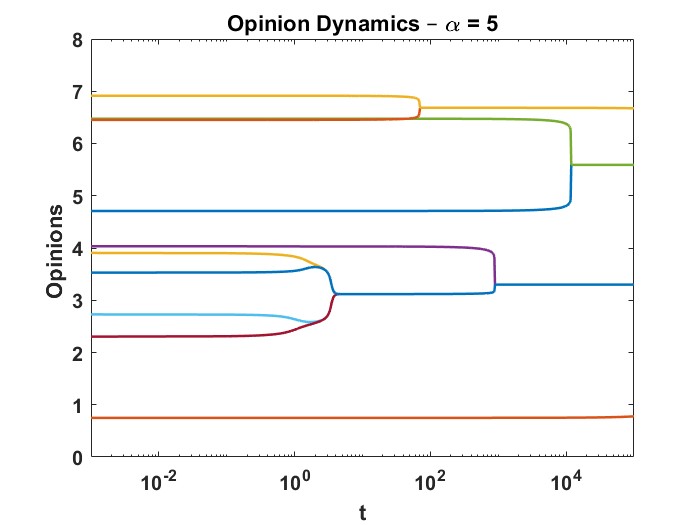}
 \includegraphics[width=0.4\textwidth]{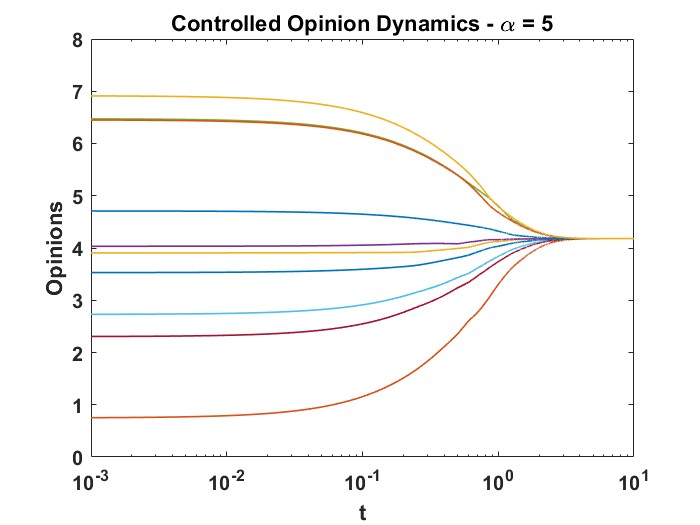}
 \includegraphics[width=0.4\textwidth]{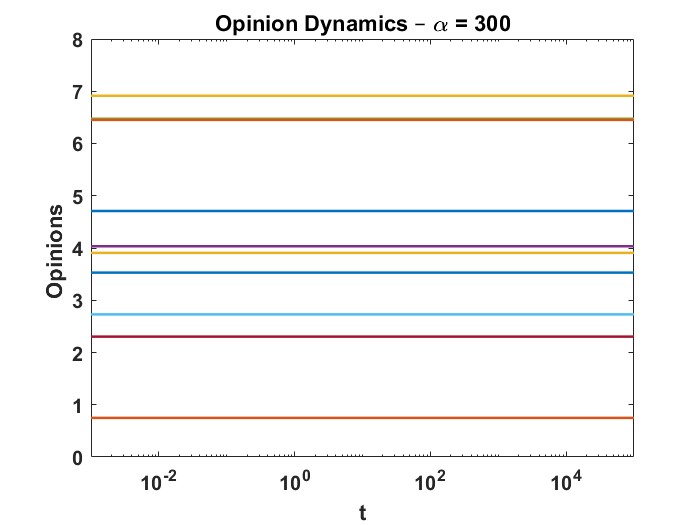}
 \includegraphics[width=0.4\textwidth]{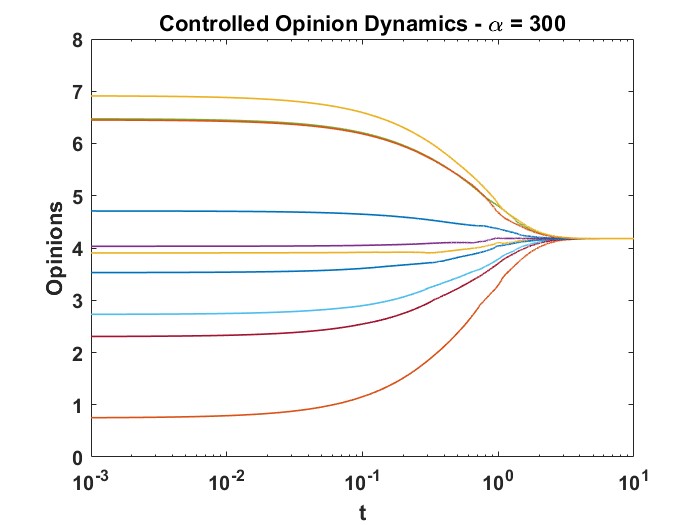}
    \caption{Opinion dynamics for four different values of the smoothness parameter $\alpha$. Left panel: uncontrolled dynamics. As $\alpha$ increases, the interaction function becomes more localized, transitioning from global consensus to the persistence of multiple opinion clusters. Right panel: controlled dynamics. Larger values of $\alpha$ induce more localized interactions, slowing down convergence and requiring more control effort to reach consensus.}
    \label{fig:alpha}
\end{figure}

To assess how the proposed optimal control strategy enforces consensus, we apply it under different interaction regimes defined by the parameter $\alpha$. This allows us to evaluate its performance in settings where the uncontrolled dynamics may naturally lead to either consensus or long-lasting opinion fragmentation. In the right panels of Figure \ref{fig:alpha}, we show the trajectories of the controlled systems. 
When applying the proposed control strategy, the agents are driven toward consensus, with the mean alignment error dropping below a prescribed threshold.


 We conducted a series of simulations with $\alpha = 1.6$, systematically varying the state dimension $d \in \{50, 100, 150\}$ and the number of agents $N \in \{50, 100, 150\}$. 
 Table \ref{tab:consensus_times_full} presents the normalized CPU time required to reach consensus, calculated relative to the baseline configuration ($d=50$, $N=50$). The data show that the computational effort increases significantly as $N$ or $d$ increases, but the sensitivity to $N$ is much stronger. 

In particular, for fixed $d$, the CPU time exhibits a super-linear growth as $N$ increases. 
On the other hand, for fixed $N$, the growth of CPU time is approximately linear. Hence, while both the number of agents $N$ and the dimensionality $d$ affect performance, the impact of $N$ is substantially greater, suggesting that scalability with respect to the population size is the primary computational bottleneck.

These results highlight the limitations of directly applying optimal control strategies in high-dimensional multi-agent systems. To ensure scalability and computational efficiency, it is essential to incorporate model order reduction techniques. Moreover, the results in Table \ref{tab:consensus_times_full} suggest that efforts to improve the scalability of the consensus algorithm should primarily focus on mitigating the computational impact of increasing the number of agents $N$. In the following section, we address this issue.

\begin{table}[b]
\centering
\caption{Normalized CPU time to reach consensus for $\alpha = 1.6$, computed as the ratio with respect to the baseline configuration $(d = 50,\ N = 50)$.}
\label{tab:consensus_times_full}
\begin{tabular}{ccccc}
\toprule
& & \multicolumn{3}{c}{\textbf{N}} \\
\cmidrule{3-5}
& & $\mathbf{50}$ & $\mathbf{100}$ & $\mathbf{150}$ \\
\midrule
\multirow{3}{*}{\textbf{d}} 
& $\mathbf{50}$  & $1\phantom{xx}$  & $4.21$ & $\phantom{x}8.76$ \\
& $\mathbf{100}$  & $1.4\phantom{x}$  & $5.27$ & $11.57$ \\
& $\mathbf{150}$  & $1.79$  & $7.04$ & $15.72$ \\
\bottomrule
\end{tabular}
\end{table}

\section{Reduction of collective dynamics: number of agents}
\label{sec:red_agents}

Since the goal is to control an ABM, having a large number of agents makes the control problem computationally expensive, even if the dimension of each agent is very low. 
Therefore, the aim of this Section is to reduce the number of agents by forming clusters based on the ABM \eqref{abm}. Each cluster will have its own dynamics determined by the center of mass; then, we will project it into a subspace and control the reduced dynamical system.
To form clusters, 
we consider the solutions $x_i(\bar{t})$ of the full model \eqref{abm} at a given time $\bar{t}$. We then apply the DBSCAN algorithm, a density-based clustering technique originally proposed in \cite{dbscan}, which allows us to obtain $K$ clusters. In contrast to the $k$-means algorithm, which requires the number of clusters as an input, DBSCAN determines the number of clusters $K$ automatically, based on the density of points in the dataset. This is achieved by identifying regions in which agents are densely packed, using two main parameters: the radius of the neighborhood $\varepsilon$ and the minimum number of points MinPts required to form a cluster.
Having applied the DBSCAN algorithm to the set of agent opinions, we obtain $K$ distinct clusters. These can be denoted by
 \begin{equation*}
 \{x_i\}_{i \in I_1}, \, \{x_i\}_{i \in I_2}, \dots, \{x_i\}_{i \in I_K}
 \end{equation*}
where each index set $I_1, I_2, \dots, I_K \subset I = \{1,2,\dots,N\}$ identifies agents belonging to a specific cluster.
For each cluster, we want to control the dynamics of the center of mass. Therefore, let us consider
\begin{equation}
    \label{center_of_mass}
    \hat{x}_l = \frac{1}{N_l} \sum_{i \in I_l} x_i
\end{equation}
the center of mass of the $l$-th cluster.
First, we want to describe how $\hat{x}_l$ evolves in time for $l = 1, 2, \dots, K$. Taking into account the time evolution of the system \eqref{abm} without control, we compute the time derivative of $\dot{\hat{x}}_l$
\begin{equation*}
\begin{aligned}
    \dot{\hat{x}}_l = \frac{1}{N_l} \sum_{i \in I_l} \dot{x}_i &= \frac{1}{N_l} \sum_{i \in I_l} \left(\frac{1}{N} \sum_{j=1}^{N} \phi(\|x_i-x_j\|) \left(x_j-x_i\right) \right) \\
    &= \frac{1}{N_l} \sum_{i \in I_l} \left(\frac{1}{N} \sum_{m=1}^{K} \sum_{j \in I_m} \phi(\|x_i-x_j\|) \left(x_j-x_i\right) \right) \\
    &= \frac{1}{N_l} \frac{1}{N} \sum_{m=1}^{K} \sum_{i \in I_l} \sum_{j \in I_m} \phi(\|x_i-x_j\|) \left(x_j-x_i\right)
    \end{aligned}
\end{equation*}
for $l = 1, 2, \dots, K$.
Since the DBSCAN algorithm forms clusters based on distance, and we are assuming that the interaction kernel $\phi$ depends on the relative distance between agents, we can suppose
\begin{equation}
    \label{assumption_P_cluster}
    \phi(\|x_i-x_j\|) \approx \phi(\|\hat{x}_l-\hat{x}_m\|), \quad \forall i \in I_l,\, \forall j \in I_m,
\end{equation}
and we obtain
\begin{equation*}
    \begin{aligned}
    \dot{\hat{x}}_l &\approx \frac{1}{N_l} \frac{1}{N} \sum_{m=1}^{K} \phi(\|\hat{x}_l-\hat{x}_m\|) \sum_{i \in I_l} \sum_{j \in I_m}  \left(x_j-x_i\right) \\
    &= \frac{1}{N_l} \frac{1}{N} \sum_{m=1}^{K} \phi(\|\hat{x}_l-\hat{x}_m\|) \sum_{i \in I_l} \left(\sum_{j \in I_m}  x_j-\sum_{j \in I_m} x_i\right) \\
    &= \frac{1}{N_l} \frac{1}{N} \sum_{m=1}^{K} \phi(\|\hat{x}_l-\hat{x}_m\|) \sum_{i \in I_l} \left( N_m \hat{x}_m - N_m x_i\right) \\
    &= \frac{1}{N_l} \frac{1}{N} \sum_{m=1}^{K} N_m \phi(\|\hat{x}_l-\hat{x}_m\|) \sum_{i \in I_l} \left( \hat{x}_m - x_i\right) \\
    &= \frac{1}{N_l} \frac{1}{N} \sum_{m=1}^{K} N_m \phi(\|\hat{x}_l-\hat{x}_m\|) \left(N_l  \hat{x}_m - N_l \hat{x}_l\right) \\
    &= \frac{1}{N} \sum_{m=1}^{K} N_m \phi(\|\hat{x}_l-\hat{x}_m\|) \left(  \hat{x}_m - \hat{x}_l\right). 
    \end{aligned}
\end{equation*}
Finally, the dynamics of the center of mass of each cluster is given by
\begin{equation}
    \label{dyn_com}
    \dot{\hat{x}}_l \approx \frac{1}{N} \sum_{m=1}^{K} N_m \phi(\|\hat{x}_l-\hat{x}_m\|) \left(  \hat{x}_m - \hat{x}_l\right)
\end{equation}
with $\hat{x}_l \in \mathbb{R}^d$ for $l = 1,2,\dots,K$. We observe that it has the same form as the full \eqref{abm} ABM, except for the \emph{weights} $N_m$, which represents the number of agents within each $m$-th cluster. These weights scale the interaction terms in the ABM, determining the influence of each cluster on the overall system's behavior and accounting for the different sizes of the clusters.

\subsection{Effect of reduction of the number of agents in the opinion-dynamics model}
In this Section, we investigate the impact of reducing the number of agents in the opinion dynamics model through clustering. We consider $\alpha = 1.6$ in the Hegselmann-Krause function; then, we fix the dimension of each agent to $d = 50$ and vary the total number of agents $N \in \{50, 100, 150\}$. The initial conditions are pre-clustered 
with agents grouped in such a way that the mean opinion distribution is deliberately non-symmetric. This setup ensures that the clusters reflect heterogeneous initial biases, avoiding artificial symmetry in the consensus trajectory.
The clustering algorithm (DBSCAN) is applied to group agents based on their spatial proximity, and the dynamics of each cluster is governed by the evolution of its center of mass, as described in Equation \eqref{dyn_com}.

Figure \ref{fig:test_cluster} illustrates the behaviour of the system over time. The first row shows the evolution of the mean opinion value, confirming that consensus is eventually reached across all configurations. The second row displays the number of clusters over time, which decreases as agents converge toward a common opinion. This reduction in the number of clusters is a direct consequence of the consensus formation process, in which initially distinct opinion groups merge.

\begin{figure}[tbp]
    \centering
    \includegraphics[width=0.33\linewidth]{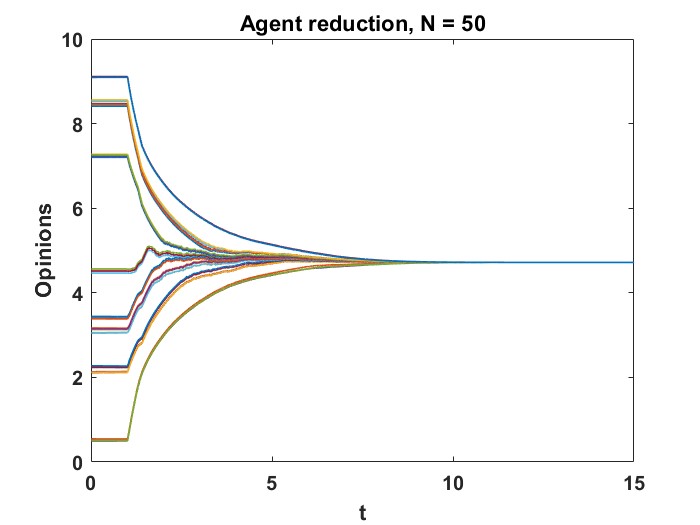}
    \includegraphics[width=0.33\linewidth]{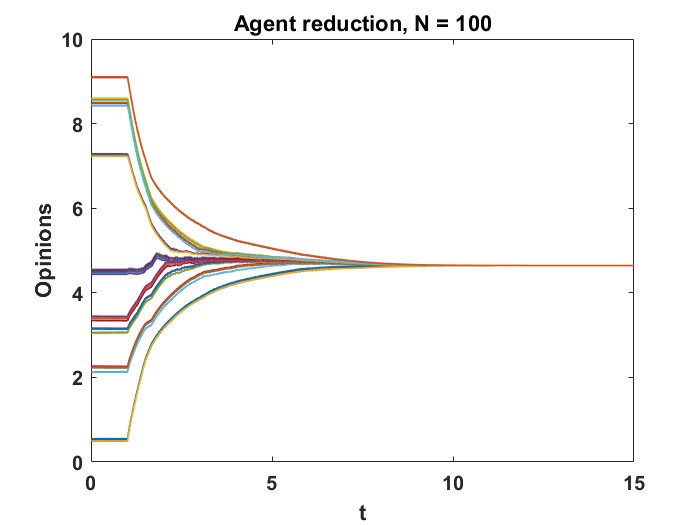}
    \includegraphics[width=0.33\linewidth]{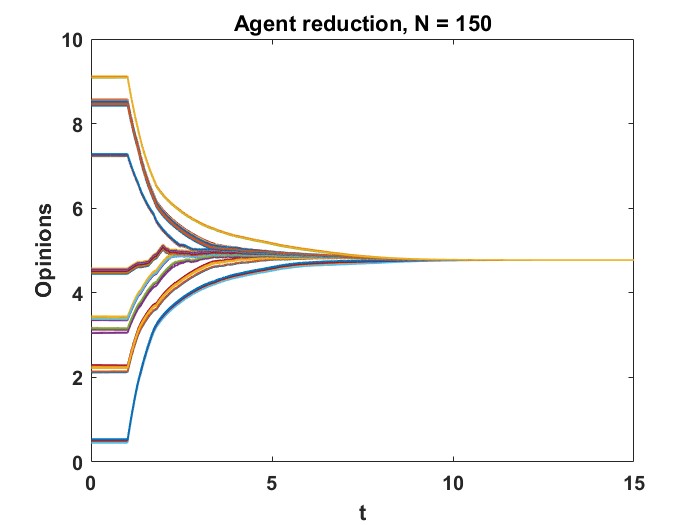}
    \includegraphics[width=0.33\linewidth]{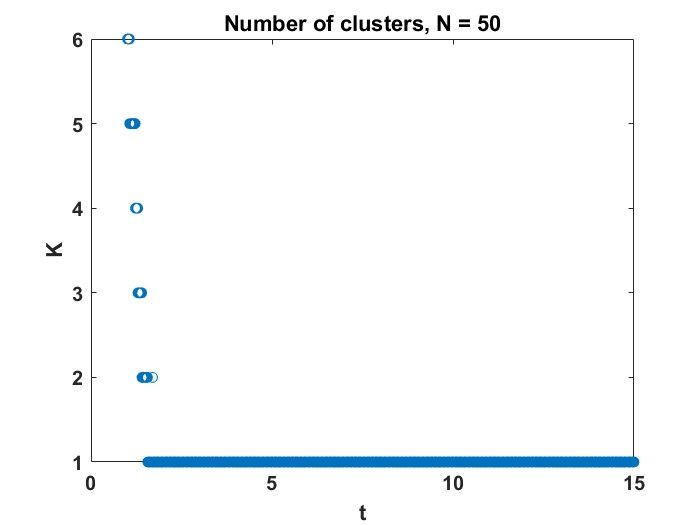}
    \includegraphics[width=0.33\linewidth]{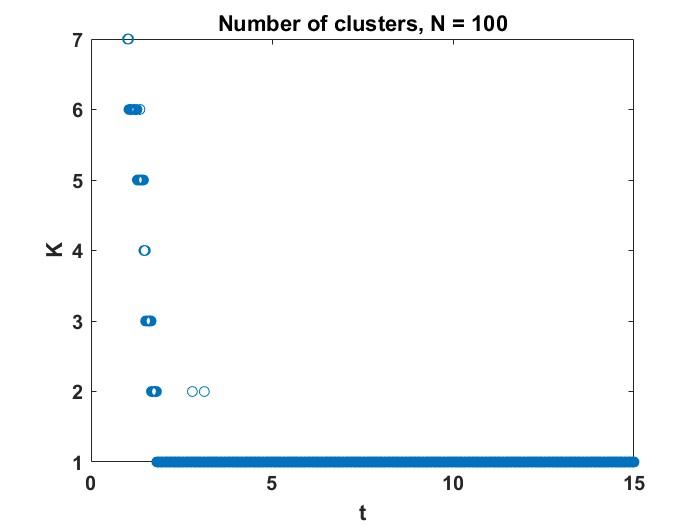}
    \includegraphics[width=0.33\linewidth]{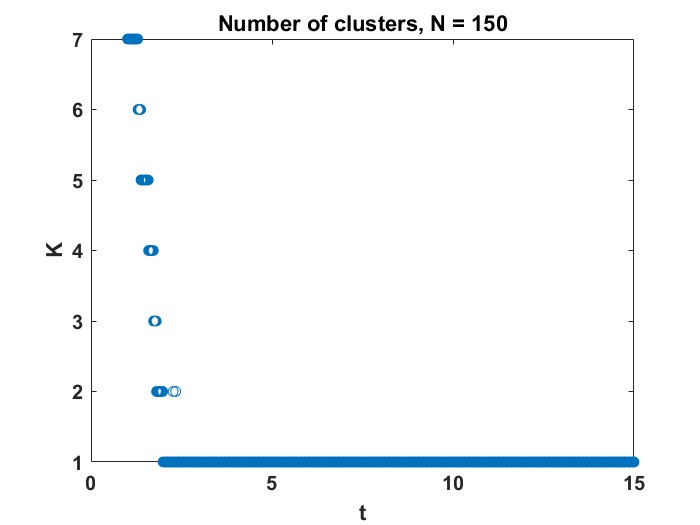}
    \caption{Effect of agent clustering on consensus formation for fixed agent dimension $d = 50$ and varying number of agents $N \in \{50, 100, 150\}$. The top row shows the evolution of the mean opinion value over time, indicating convergence to consensus. The bottom row displays the number of clusters over time, which decreases as agents merge into unified opinion groups. All simulations start from $K \ll N$ initial clusters, reflecting the pre-clustered initial conditions.}
    \label{fig:test_cluster}
\end{figure}
Table \ref{tab:consensus_cpu_cluster} illustrates the computational gains achieved by reducing the number of agents through clustering. Specifically, it reports the speed-up factors, defined as the ratio of the CPU time required by the full agent-based model to that required by the reduced model, for various values of $d$ and $N$.

The results show that the speed-up factor increases considerably with the number of agents $N$. For instance, when $d = 50$, the speed-up factor rises from approximately 61 for $N = 50$ to about $239$ for $N = 150$. This pattern is consistent across all values of $d$, indicating that the computational savings grow significantly as the population size increases.

The effect of increasing $d$ for fixed $N$ is more nuanced. For a given $N$, the speed-up factors decrease moderately as $d$ increases, suggesting that the benefits of agent reduction become slightly less pronounced in higher-dimensional settings, although they remain substantial.

Overall, the table confirms that clustering offers significant computational advantages, particularly for large-scale systems with many agents. The speed-up factors exceeding $100$ in several configurations highlight the potential of this approach to make otherwise intractable simulations feasible.


\begin{table}[htbp]
\centering
\caption{Effect of agent reduction on the CPU time required to reach consensus for $\alpha = 1.6$. The table reports the speed-up factor obtained by reducing the number of agents.}
\label{tab:consensus_cpu_cluster}
\begin{tabular}{ccccc}
\toprule
& & \multicolumn{3}{c}{\textbf{N}} \\
\cmidrule{3-5}
& & $\mathbf{50}$ & $\mathbf{100}$ & $\mathbf{150}$ \\
\midrule
\multirow{3}{*}{\textbf{d}} 
& $\mathbf{50}$ & $61.36$ & $166.62$ & $239.16$ \\
& $\mathbf{100}$  & $46.13$ & $112.19$ & $167.23$ \\
& $\mathbf{150}$ & $37.55$ & $\phantom{x}96.85$ & $145.07$ \\
\bottomrule
\end{tabular}
\end{table}

\section{Reduction of collective dynamics: dimension of agents}
\label{sec:red_dim}
This Section primarily aims to explore the reduction in the dimensionality of each agent, a goal that can be achieved through various methods \cite{Sirovich,benner2015survey,brunton2022data,hesthaven2022reduced,Bongini2015}. 
Among these, we will consider the application of the POD method that is described in the next Section.

\subsection{Dimension reduction by Proper Orthogonal Decomposition}
\label{sec:red_pod}
In this Section we apply the Proper Orthogonal Decomposition (POD) with a Galerkin projection, to construct a reduced ABM, with $N$ agents with states $x^r_i \in \mathbb{R}^r$, with $r \ll d$. We aim to reduce the dimension $d$ of each agent, by applying POD with a Galerkin projection to each full system \eqref{abm}. First of all, we need to construct a subspace onto which project the equations of \eqref{abm}. We look for a matrix $\Psi_r \in \mathbb{R}^{d \times r}$, which will allow us to preserve distances and, as a consequence, to obtain for the reduced ABM the same structure of \eqref{abm}. 
To obtain the matrix $\Psi_r$, we first construct a snapshot matrix $S$. We collect data from \eqref{abm}, an approximate solution $\{x_i(t_k)\}$, $k = 0,1,\ldots, n$ for some time instances $\{t_0,t_1, \ldots, t_n\}$. We build the snapshot matrix as follows
\begin{equation}
    \label{snapshots}
    S = \left[S_1, S_2, \ldots, S_N \right]
\end{equation}
with $S_i = \left[ x_i(t_0), \, x_i(t_1), \, \ldots, \, x_i(t_n) \right] \in \mathbb{R}^{d \times (n+1)}$, thus $S \in \mathbb{R}^{d \times N(n+1)}$. The matrix $\Psi_r$, i.e. the POD basis of rank $r$ we are looking for, is given by the left singular vectors of the Singular Value Decomposition (SVD) of $S$, $S \approx \Psi_r \Sigma_r W_r^T$, where $\Psi_r \in \mathbb{R}^{d \times r}, \Sigma_r \in \mathbb{R}^{r \times r}$ and $W_r \in \mathbb{R}^{N(n+1) \times r}$.
We apply POD with a Galerkin projection, by making the ansatz
\begin{equation}
    \label{ansatz_x}
    x_i \approx \Psi_r x_i^r
\end{equation}
\begin{equation}
    \label{ansatz_u}
    u_i \approx \Psi_r u_i^r
\end{equation}
for the state and control variable, where $x_i^r, u_i^r \in \mathbb{R}^r$ are the position (respectively control) in the reduced space of dimension $r$.
By plugging the assumptions \eqref{ansatz_x}-\eqref{ansatz_u} in the full agent-based model \eqref{abm}, we obtain
\begin{equation}
    \Psi_r \dot{x}_i^r = \frac{1}{N} \sum_{j=1}^{N} \phi(\|\Psi_r x_i^r -\Psi_r x_j^r\|) \left(\Psi_r x_j^r-\Psi_r x_i^r\right) + \Psi_r u_i^r.
\end{equation}
We multiply by $\Psi_r^T$ and employ the orthogonality, obtaining the reduced system,
\begin{equation*}
      \begin{aligned}
      \dot{x}_i^r &= \frac{1}{N} \Psi_r^T \sum_{j=1}^{N} \phi(\|\Psi_r x_i^r-\Psi_r x_j^r\|)\left(\Psi_r x_j^r- \Psi_r x_i^r\right) + u_i^r \\
      &= \frac{1}{N}  \sum_{j=1}^{N} \phi(\|\Psi_r x_i^r-\Psi_r x_j^r\|) \, \Psi_r^T \left(\Psi_r x_j^r- \Psi_r x_i^r\right) + u_i^r\\
      &= \frac{1}{N}  \sum_{j=1}^{N} \phi(\|\Psi_r x_i^r-\Psi_r x_j^r\|) \left( x_j^r- x_i^r\right) + u_i^r
      \end{aligned}
\end{equation*}
with unknowns $\x_i^r, \v_i^r \in \mathbb{R}^r$ and $r \ll d$. \\
We observe that the interaction kernel $\phi$ is evaluated at $\Psi_r x_i^r \in \mathbb{R}^d$, thus it still depends on the dimension of the full model. 
Since $\Psi_r$ preserves distances, we have that
\begin{equation*}
     \phi(\|\Psi_r x_i^r - \Psi_r x_j^r\|) = \phi(\|x_i^r-x_j^r\|)
\end{equation*}
where $\|x_i^r-x_j^r\|$ is the Euclidean distance of vectors in $\mathbb{R}^r$. 
Therefore, the reduced model
\begin{equation}
    \label{pod_abm}
    \dot{x}_i^r = \frac{1}{N} \sum_{j=1}^{N} \phi(\|x_i^r-x_j^r\|) \left(x_j^r-x_i^r\right) + u_i^r
\end{equation}
represents the dynamics of the ABM in the reduced space, while keeping the number of agents fixed. This allows us to obtain a reduced model \eqref{pod_abm} that has the same structure as \eqref{abm}. Our goal is to achieve consensus within the reduced system through control, ensuring that even when returning to the full space, the system tends to consensus. To this aim, we denote by $\widetilde{x}_i(t) \in \mathbb{R}^d$ the solution obtained by the reduced model and reconstructed by \eqref{ansatz_x}.
The following result holds.


\begin{tr}
     If the reduced system \eqref{pod_abm} tends to consensus, then it holds 
     \begin{equation}
         \lim_{t \to + \infty}\| \widetilde{x}_i(t) - \bar{\widetilde{x}}(t)\| = 0
     \end{equation}
with 
\begin{equation}
    \bar{\widetilde{x}}(t) =\frac{1}{N} \sum_{i=1}^{N} \widetilde{x}_i(t).
\end{equation}
In other words, the reconstructed state variable also tends towards consensus.
\end{tr}
\begin{proof}
We suppose that system \eqref{pod_abm} tends to consensus, that is
     $$\lim_{t \to + \infty}\| x_i^r(t) - \bar{x}^r(t)\| = 0$$
     with $\bar{x}^r(t) = \frac{1}{N}\sum_{i=1}^N x_i^r(t)$. 
     We have that
\begin{equation*}
\begin{aligned}
\lim_{t \to + \infty}\| \widetilde{x}_i(t) - \bar{\widetilde{x}}\| &=  \lim_{t \to + \infty}\left \| \widetilde{x}_i(t) - \frac{1}{N} \sum_{i=1}^{N} \widetilde{x}_i(t) \right \| = \lim_{t \to + \infty} \left \| \Psi_r x_i^r(t) - \frac{1}{N} \sum_{i=1}^{N} \Psi_r x_i^r(t) \right \| \\
&= 
\lim_{t \to + \infty}\left \| \Psi_r \left( x_i^r(t) - \frac{1}{N} \sum_{i=1}^{N} x_i^r(t) \right) \right \| = \lim_{t \to + \infty}\left \| x_i^r(t) - \frac{1}{N} \sum_{i=1}^{N} x_i^r(t) \right \| = 0,
\end{aligned}
\end{equation*}
where the last two equalities follow from the orthogonality of the columns of $\Psi_r$ and from the hypothesis of consensus for \eqref{pod_abm}.
\end{proof}

\subsection{Reduction of agents' dimension in the opinion-dynamics model}
\begin{figure}[t]
    \centering
    \includegraphics[width=0.33\linewidth]{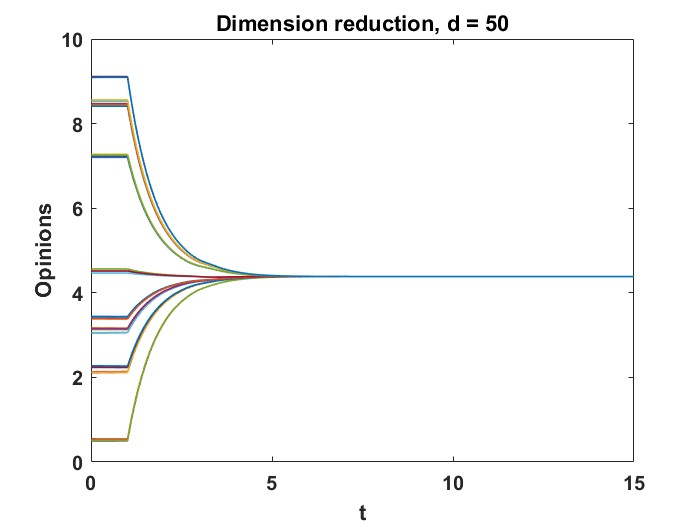}
    \includegraphics[width=0.33\linewidth]{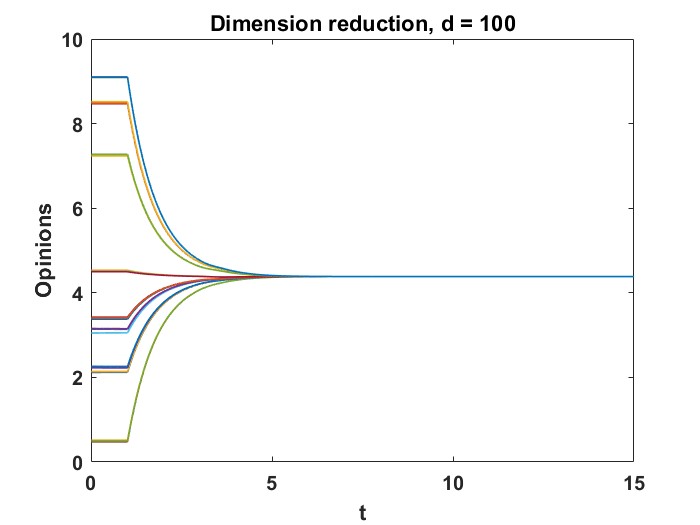}
    \includegraphics[width=0.33\linewidth]{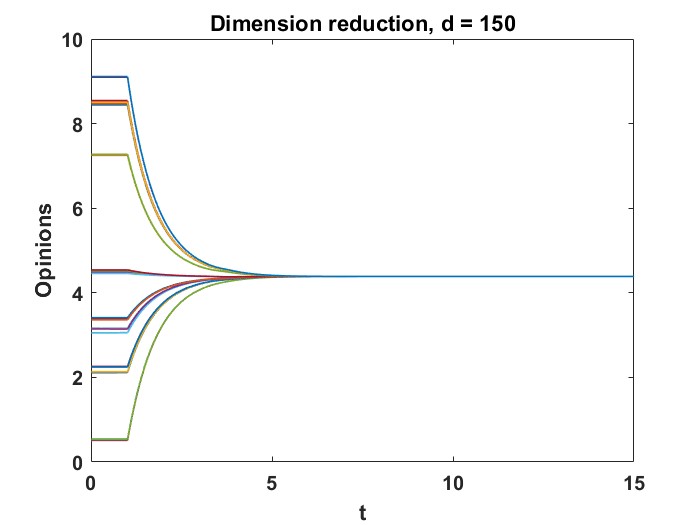}
    \includegraphics[width=0.33\linewidth]{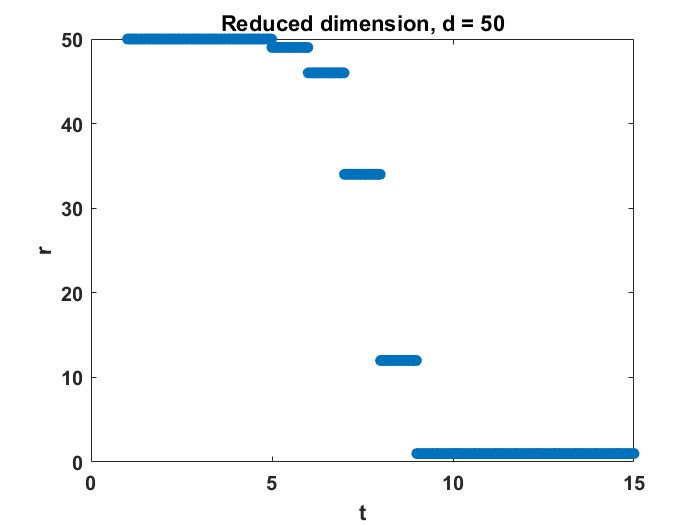}
    \includegraphics[width=0.33\linewidth]{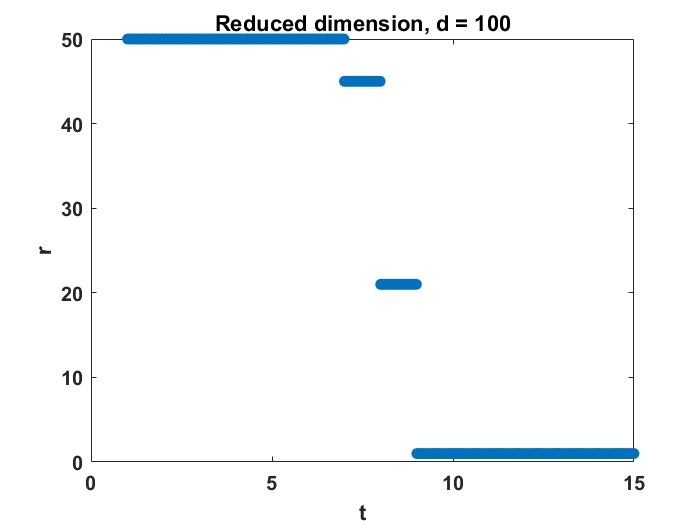}
    \includegraphics[width=0.33\linewidth]{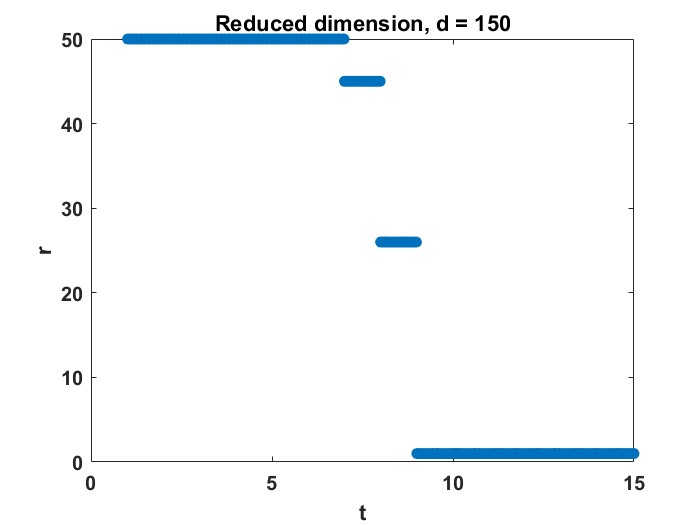}
    \caption{Effect of POD-based dimension reduction on consensus formation for fixed number of agents $N = 50$ and varying agent dimension $d \in \{50, 100, 150\}$. The top row shows the evolution of the mean opinion over time, confirming convergence to consensus. The bottom row reports the effective reduced dimension $r$ selected via singular value decomposition with a tolerance of $10^{-3}$. As the system approaches consensus, the rank of the snapshot matrix decreases, resulting in lower-dimensional representations.}
    \label{fig:test_pod}
\end{figure}

We now turn our attention to the reduction of the dimensionality of each agent in the opinion dynamics model. Unlike in the previous section, where the number of agents was reduced via clustering, here we maintain the full agent population and instead project the high-dimensional ABM onto a lower-dimensional subspace. This is achieved using the POD method with Galerkin projection, as described in Section \ref{sec:red_pod}.

We fix the number of agents to $N = 50$ and vary the state dimension $d \in \{50, 100, 150\}$, applying POD to obtain reduced models of dimension $r \ll d$. The reduced dynamics preserves the structure of the original ABM, while lowering the computational cost. The POD basis is constructed from snapshots of the full system, and the reduced system evolves in the subspace spanned by the most significant modes.

Figure \ref{fig:test_pod} shows the results of the simulations. The top row displays the evolution of the mean opinion over time, confirming that consensus is achieved even in the reconstructed agents. The bottom row reports the effective reduced dimension $r$ selected based on a tolerance threshold in the singular values. Interestingly, the value of $r$ tends to decrease as the system approaches consensus. This behavior is attributed to the redundancy in the snapshot data: as agents align their opinions, the variability across snapshots diminishes, leading to a lower-rank representation.

Table \ref{tab:consensus_cpu_pod} presents the speed-up factors achieved by reducing the dimension of agents, comparing the CPU time required by the reduced and full-order models for various values of $d$ and $N$. The results show that dimensionality reduction yields modest computational savings overall and the benefits become more evident as $d$ increases. In particular, for $d = 50$, the speed-up factors remain close to $1$, indicating virtually no reduction in computational cost. For $d = 100$, the speed-up improves slightly, reaching values around $1.1$–$1.25$, while for $d = 150$, the speed-up reaches approximately $1.3$–$1.6$ depending on $N$. These results confirm that POD-based dimensionality reduction is most beneficial for systems with high-dimensional agents, where the relative reduction in computational effort becomes more substantial. However, the overall gains remain moderate compared to the impact of reducing the number of agents, as shown previously.

\begin{table}[tbp]
\centering
\caption{Effect of POD-based dimension reduction on the CPU time required to reach consensus for $\alpha = 1.6$. The table reports the speed-up factor obtained by reducing the agents' dimension.}
\label{tab:consensus_cpu_pod}
\begin{tabular}{ccccc}
\toprule
& & \multicolumn{3}{c}{\textbf{N}} \\
\cmidrule{3-5}
& & $\mathbf{50}$ & $\mathbf{100}$ & $\mathbf{150}$ \\
\midrule
\multirow{3}{*}{\textbf{d}} 
& $\mathbf{50}$ & $0.99$ & $1.04$ & $0.99$ \\
& $\mathbf{100}$  & $1.25$ & $1.09$ & $1.24$ \\
& $\mathbf{150}$ & $1.62$ & $1.35$ & $1.15$ \\
\bottomrule
\end{tabular}
\end{table}

\section{Two-level reduction framework}
\label{sec:framework}
In this Section, we introduce the framework designed to induce consensus through the application of optimal control to the clustered reduced ABM. Specifically, we outline the methodology employed to structure the clustering process, reduce the complexity of the system, and implement control strategies that guide the dynamics of the reduced model toward a consensus state.

The process begins at iteration $k = 1$ with an initial simulation of the full ABM where no control is applied, allowing the agents to evolve naturally according to their governing dynamics over the time interval $[t_0, t_n]$, discretized using a step size $h_t$ such that $t_n = t_0 + n h_t$. 
Once the uncontrolled evolution has been simulated, clustering techniques are employed to group agents with similar dynamics. Specifically, the DBSCAN algorithm is used at time $t_{n+k-1}$ to classify agents into $K$ clusters according to their spatial distribution. 
After clustering, we derive the dynamics of their centers of mass, following the formulation given in Equation \eqref{dyn_com}. 

The next step involves the construction of the snapshot matrix $S$, as defined in equation \eqref{snapshots} in the time interval $[t_{k-1},t_{n+k-1}]$. 
This matrix is used for dimensionality reduction using the POD method, recalled in Section \ref{sec:red_pod}. 
By applying POD to the system \eqref{dyn_com}, we obtain 
\begin{equation}
    \label{com_reduced}
    \dot{\hat{x}}_l^r = \frac{1}{N} \sum_{m=1}^{K} N_m \phi(\|\hat{x}_l^r-\hat{x}_m^r\|) \left(  \hat{x}_m^r - \hat{x}_l^r\right)
\end{equation}
a reduced-order model that preserves the essential characteristics of the system \eqref{dyn_com} while significantly decreasing computational complexity.

We then formulate an optimal control problem
\begin{equation}
    \label{com_cost}
    \min_{\hat{u}^r} \int_{t_{n+k-1}}^{t_{\nu(n+k-1)}} \frac{1}{K} \sum_{j=1}^{K} \left(\|\hat{x}_j^r(t) - \bar{\hat{x}}^r(t)\|^2 + \gamma \|\, \hat{u}_j^r(t)\|^2\right) dt
\end{equation}
subject to the controlled system dynamics
\begin{equation}
    \label{com_control}
    \dot{\hat{x}}_i^r = \frac{1}{N} \sum_{j=1}^{K} N_j \, \phi(\|\hat{x}_i^r-\hat{x}_j^r\|) \left(  \hat{x}_j^r - \hat{x}_i^r \right) + \hat{u}_i^r
\end{equation}
The aim consists in finding the control inputs $\hat{u}_i^r \in \mathbb{R}^r$ by solving the optimal control problem \eqref{com_cost}, 
subject to the constraints in \eqref{com_control}, over the time interval $[t_{n+k-1}, t_{\nu(n+k-1)}]$. 
The resulting optimal controls are then assigned to the individual agents within each cluster, setting $u_i^r(t_{n+k}) = \hat{u}_l^r(t_{n+k})$ for every agent $i$ belonging to cluster $l$.

After establishing the reduced controls, we project them back into the full-dimensional space using the transformation given by equation \eqref{ansatz_u}, 
thereby obtaining the corresponding full-space controls $u_i(t_{n+k})$. With these control inputs, we then compute the next state of each agent, $x_i(t_{n+k})$, reflecting the influence of the newly applied control strategies.

This iterative process continues until the system reaches a consensus state, indicating that the control has successfully guided the agents toward a desired configuration.
We summarize these steps in Algorithm \ref{alg:reduction-control}, which provides a structured representation of the methodology described above. A visual synthesis of the framework is also provided in Figure \ref{fig:flowdiagram}.

\begin{algorithm}[ht]
\caption{Adaptive Control via two-level reduction: clustering and POD}
\label{alg:reduction-control}
\begin{algorithmic}[1]
\State \textbf{INPUT:} threshold $tol$, $h_t$
\State \textbf{OUTPUT:} $x_i(t)$, $K(t)$, $r(t)$, $u_i(t)$, $T$
\State Simulate the uncontrolled ABM \eqref{abm} in $[t_0, t_n]$ with step size $h_t$
\State Set $k = 1$
\While{$X > tol$}
    \State Cluster the agents $\{x_i(t_{n+k-1})\}_{i=1}^N$ using DBSCAN to obtain $K$ clusters
    \State Compute the center of mass dynamics for each cluster via \eqref{dyn_com}
    \State Construct the snapshot matrix $S$ as in \eqref{snapshots} in $[t_{k-1},t_{n+k-1}]$
    \State Apply POD obtain a reduced model of dimension $r$ \eqref{com_reduced}
    \State Solve the reduced optimal control problem \eqref{com_cost}--\eqref{com_control} on $[t_{n+k-1}, t_{\nu(n+k-1)}]$
    \For{each cluster $l = 1, \dots, K$}
        \For{each agent $i \in I_l$}
            \State Set reduced control $u_i^r(t_{n+k}) \gets \hat{u}_l^r(t_{n+k})$
            \State Lift control: $u_i(t_{n+k}) \gets$ projection via \eqref{ansatz_u}
        \EndFor
    \EndFor
    \State Compute $x_i(t_{n+k})$ using $u_i(t_{n+k})$
    \State Set $k = k+1$
    \State Go to Step 6
\EndWhile
\end{algorithmic}
\end{algorithm}

\textit{Inputs.} The inputs of the Algorithm are the time step $h_t$ and the threshold $tol$ for the consensus parameter $X$ defined in \eqref{consensus_parameter}.

\textit{Initialization.} We simulate the uncontrolled ABM \eqref{abm} in the time interval $[t_0,t_n]$ with time step $h_t$, where $t_n = t_0 + n h_t$.

\textit{Outputs.} The Algorithm returns the opinion values $x_i$, the controls $u_i$, the number of clusters $K$ and the reduced dimension $r$ over time, until the final time $T$ when the system achieves consensus.

\begin{remark}[Inputs for the DBSCAN algorithm]
The DBSCAN algorithm requires two input parameters: the radius $\varepsilon$ that defines the neighborhood of a point, and the minimum number of points $\text{MinPts}$ needed to form a dense region. In our setting, we fix $\text{MinPts} = 1$.
We adopt an empirical strategy based on the global scale of the initial data. Specifically, we set $\varepsilon = \|\text{data}\| / N$, where $\|\text{data}\|$ denotes the Frobenius norm of the matrix containing the initial positions of all agents, and $N$ is the total number of agents. Geometrically, this corresponds to an average measure of the spread of the data in the opinion space. Since the Frobenius norm aggregates the squared distances of all agents from the origin, dividing by $N$ yields a normalized estimate of the typical inter-agent distance. This choice ensures that agents within reasonable proximity are grouped together, while agents from distinct opinion clusters remain separated. In practice, it provides a robust and scale-adaptive threshold for density-based clustering, especially when the initial data is pre-clustered. Moreover, it avoids the need for manual tuning or heuristic distance plots, making it suitable for automated simulations across varying configurations.
\end{remark}

\begin{remark}[Choice of the reduced dimension]
The reduced dimension $r$ in the POD method is selected by analyzing the singular values of the snapshot matrix $S$. Specifically, we fix a tolerance threshold $\tau$ (e.g. $\tau = 10^{-3}$) and retain only those singular values $\sigma_i$ such that $\sigma_i \geq \tau$. This procedure ensures that the POD basis captures the dominant modes of the system while discarding directions with negligible energy. As the system evolves toward consensus, the variability among agents decreases, leading to increased redundancy in the snapshot data. Consequently, the effective rank of $S$ tends to decrease over time, allowing for lower-dimensional representations without compromising the accuracy of the reduced model.
\end{remark}

\begin{figure}[htbp]
\centering
\resizebox{\textwidth}{!}{
\begin{tikzpicture}[
    box/.style={rectangle, draw=black, thick, minimum width=4cm, minimum height=2cm, align=center, rounded corners=5pt},
    smallbox/.style={rectangle, draw=black, thick, minimum width=3.5cm, minimum height=1.5cm, align=center, rounded corners=5pt},
    arrow/.style={->, >=stealth, thick},
    label/.style={font=\small}
]


\node[box, fill=blue!10] (fullabm) at (-6,-2.5) {
    \textbf{Full ABM}\\[2pt]
    $N$ agents\\[2pt]
    $\dot{x}_i = \frac{1}{N}\sum_{j=1}^{N} \phi(\|x_i - x_j\|)(x_j - x_i)$\\[2pt]
    $x_i \in \mathbb{R}^d, \quad i = 1,\ldots,N$
};

\node[box, fill=orange!10, below=2cm of fullabm] (clustered) {
    \textbf{Clustered ABM}\\[2pt]
    $K$ clusters $(K \ll N)$\\[2pt]
    $\dot{\hat{x}}_l = \frac{1}{N}\sum_{m=1}^{K} N_m\phi(\|\hat{x}_l - \hat{x}_m\|)(\hat{x}_m - \hat{x}_l)$\\[2pt]
    $\hat{x}_l \in \mathbb{R}^d, \quad l = 1,\ldots,K$
};

\draw[arrow, color=green!60!black, line width=2pt] 
  (fullabm.south) -- 
  node[midway, left, font=\footnotesize\bfseries] {Level 1: Clustering} 
  node[midway, right, font=\footnotesize] {DBSCAN} 
  (clustered.north);

\draw[arrow, color=purple!60!black, line width=2pt] (clustered.east) -- node[above, font=\footnotesize\bfseries] {Level 2: POD} node[below, font=\footnotesize] {Dimension Reduction} ++(3,0);

\node[box, fill=yellow!10, right=3cm of clustered] (reduced) {
    \textbf{Reduced ABM}\\[2pt]
    $K$ clusters, $r$-dimensional $(r \ll d)$\\[2pt]
    $\dot{\hat{x}}^r_l = \frac{1}{N}\sum_{m=1}^{K} N_m\phi(\|\hat{x}^r_l - \hat{x}^r_m\|)(\hat{x}^r_m - \hat{x}^r_l)$\\[2pt]
    $\hat{x}^r_l \in \mathbb{R}^r, \quad l = 1,\ldots,K$
};

\node[smallbox, fill=gray!10, above=0.5cm of reduced, minimum width=4cm] (pod) {
    \textbf{POD Basis}\\[2pt]
    $S = [S_1, S_2, \ldots, S_N]$\\[2pt]
    $S \approx \Psi_r \Sigma_r W_r^T$\\[2pt]
    $x_i \approx \Psi_r x^r_i$
};

\draw[arrow, color=gray, line width=1pt] (pod.south) -- (reduced.north);

\node[box, fill=purple!10, below=2.5cm of reduced, minimum width=6cm, minimum height=3cm] (control) {
    \textbf{Optimal Control Problem}\\[4pt]
    $\min_{\hat{u}^r} \int_{t_{n+k-1}}^{t_{\nu(n+k-1)}} \frac{1}{K}\sum_{j=1}^{K} \left[\|\hat{x}^r_j(t) - \bar{\hat{x}}^r(t)\|^2 + \gamma\|\hat{u}^r_j(t)\|^2\right] dt$\\[6pt]
    subject to:\\[2pt]
    $\dot{\hat{x}}^r_i = \frac{1}{N}\sum_{j=1}^{K} N_j \phi(\|\hat{x}^r_i - \hat{x}^r_j\|)(\hat{x}^r_j - \hat{x}^r_i) + \hat{u}^r_i$\\[4pt]
    \textit{Pontryagin's Principle:} $\hat{u}^r_i = -\frac{N}{2\gamma}\hat{p}^r_i$
};

\draw[arrow, color=purple!60!black, line width=2pt] (reduced.south) -- (control.north);


\node[box, fill=green!10, below=2.5cm of clustered, minimum width=5cm] (reconstruction) {
    \textbf{Control Reconstruction}\\[2pt]
    For each agent $i \in I_l$: $u^r_i(t_{n+k}) \leftarrow \hat{u}^r_l(t_{n+k})$\\[2pt]
    Lift control: $u_i(t_{n+k}) = \Psi_r u^r_i(t_{n+k})$\\[2pt]
    Update: $x_i(t_{n+k})$ using controlled dynamics
};

\draw[arrow, color=red!60!black, line width=2pt] (control.south) |- ++(0,-1) -| node[pos=0.25, below, font=\footnotesize\bfseries] {Control Design} (reconstruction.south) -- (reconstruction);


\draw[arrow, dashed, color=gray, line width=1.5pt] (reconstruction.north) -- node[right, font=\footnotesize] {Iterate} (clustered.south);

\end{tikzpicture}
}
\caption{Flow diagram of the two-level model reduction framework for optimal control of agent-based dynamics. Agent clustering enables structural simplification, while POD captures dominant behavioral modes. The reduced model is used for control design and subsequently mapped back to the full system to iteratively refine the strategy.}
\label{fig:flowdiagram}
\end{figure}
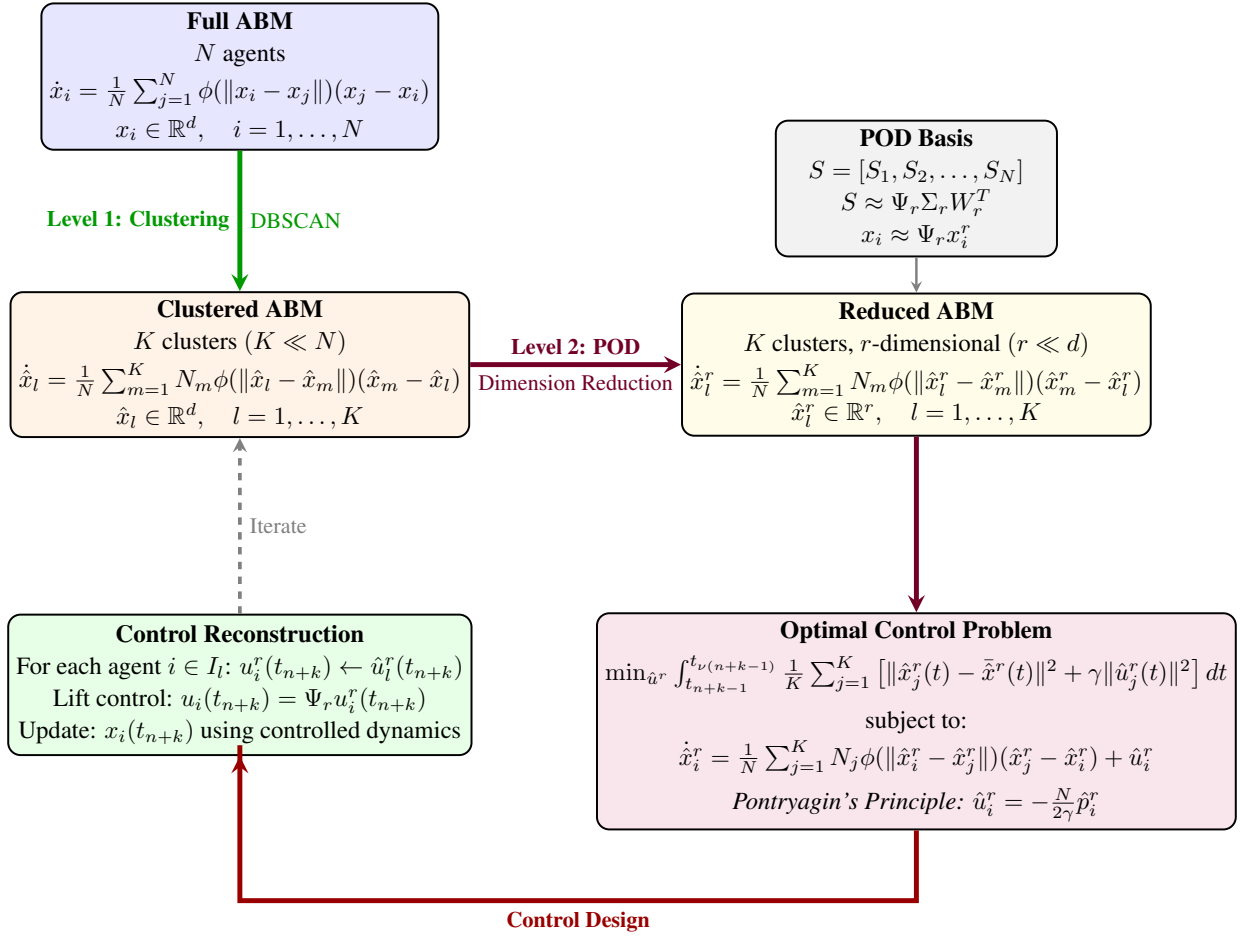

\subsection{Two-level reduction in the opinion-dynamics model}

In this Section, we assess the performance of the full reduction framework, which combines agent clustering and dimensionality reduction via POD. We consider a high-dimensional system with $N = 150$ agents, each evolving in an opinion space of dimension $d = 150$. The goal is to demonstrate that the proposed two-level reduction strategy enables efficient and accurate consensus control, even in settings where the full model would be computationally prohibitive.
The framework proceeds in two stages. First, we apply the DBSCAN algorithm to the initial agent configuration, identifying the number $K$ of clusters based on spatial density. Then we consider the dynamics of the centers of mass of these clusters, reducing the number of controlled entities from $N$ to $K$. Second, we perform POD on the trajectories of the cluster centers to reduce the state-space dimension. The reduced basis is constructed from a snapshot matrix, and the effective dimension $r$ is selected by discarding singular values below a fixed threshold $\tau = 10^{-3}$.

The top left panel of Figure \ref{fig:test_framework} shows the evolution of the mean opinion over time in the full order model, where control is applied throughout the time interval. 
In contrast, the top right panel displays the dynamics under the two-level reduction framework, which evolves uncontrolled in $[0,1]$ before switching to control. Although the controlled full order ABM reaches consensus earlier in time - just over $10$ seconds compared to approximately $30$ seconds in the framework scenario. 
The bottom left panel displays the number of clusters over time, which decreases as agents merge into unified opinion groups. The bottom right panel reports the reduced dimension $r$ selected at each time step. As the system approaches consensus, the variability in the snapshot data diminishes, resulting in a lower-rank representation.

Table \ref{tab:consensus_cpu_complete} reports the speed-up factors achieved when both agent reduction (reducing $N$) and dimensionality reduction (reducing $d$) are applied together. The speed-up is computed by comparing the runtime of the proposed two-level framework against that of the full control strategy applied to the same configuration $(d,N)$, without any form of reduction. This direct comparison allows us to quantify the computational advantage while preserving the structural consistency of the benchmark. 
The results clearly demonstrate that the two-level reduction strategy yields substantial computational gains. Speed-up factors exceed $100$ in many cases, and reach values above $200$ when $N=150$ and $d=50$.
As expected, the speed-up factor increases consistently with the number of agents $N$, confirming that the combined reduction strategy is particularly effective in large-scale scenarios. As observed previously, speed-up factors decrease slightly as $d$ increases, reflecting the cost associated with POD-based reduction. Nevertheless, even for $d=150$, the gains remain significant - for instance, the speed-up reaches approximately $135$ when $N=150$.

Despite a slightly longer convergence time—approximately $30$ seconds versus $10$ for the full-order system—the two-level reduction framework achieves better computational performance than the full-order approach. These results demonstrate the efficiency of combining clustering and dimensionality reduction techniques. The two-level reduction allows for substantial computational savings, making the simulation of high-dimensional, large-agent systems computationally feasible while preserving the system's essential dynamics.


\begin{figure}[htbp]
    \centering
    \includegraphics[width=0.4\linewidth]{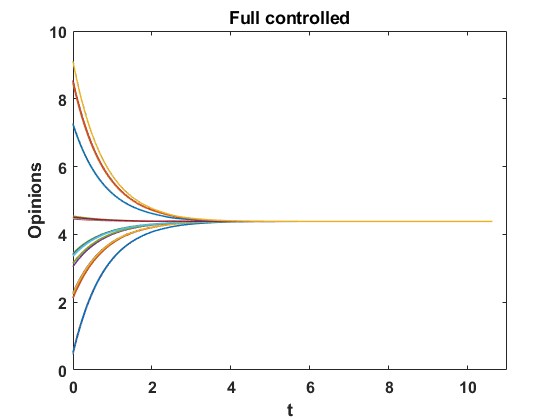}
    \includegraphics[width=0.4\linewidth]{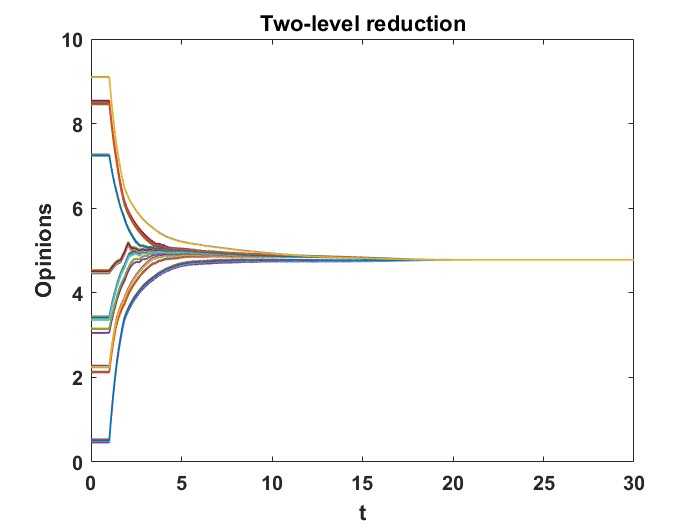}
    \includegraphics[width=0.4\linewidth]{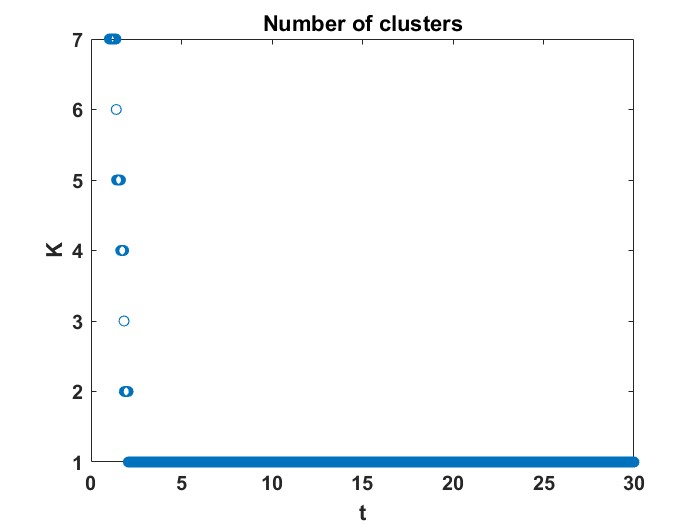}
    \includegraphics[width=0.4\linewidth]{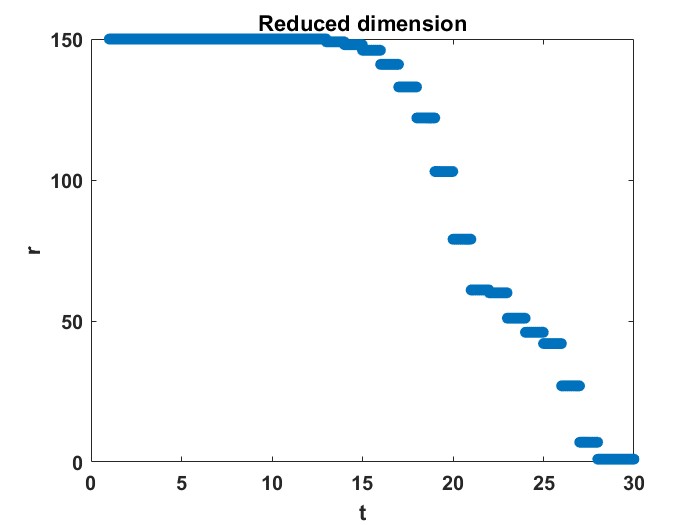}
    \caption{Two-level reduction results for $d = 150$ and $N = 150$. Top panel: evolution of the mean opinion over time, confirming convergence to consensus with a threshold of $10^{-19}$ for the full model without any reduction (left) and for the proposed reduction framework (right). Bottom panels: the left plot displays the number of clusters over time, which decreases as agents merge into unified opinion groups, whereas the right plot reports the reduced dimension $r$ selected via POD at each time step, based on a singular value threshold of $10^{-3}$. The results demonstrate the effectiveness of the combined clustering and dimensionality reduction framework in inducing consensus while significantly reducing computational complexity.}
    \label{fig:test_framework}
\end{figure}

\begin{table}[htbp]
\centering
\caption{Effect of two-level reduction on the CPU time required to reach consensus for $\alpha = 1.6$. The table reports the speed-up factor obtained by reducing both the number of agents and their dimension.}
\label{tab:consensus_cpu_complete}
\begin{tabular}{ccccc}
\toprule
& & \multicolumn{3}{c}{\textbf{N}} \\
\cmidrule{3-5}
& & $\mathbf{50}$ & $\mathbf{100}$ & $\mathbf{150}$ \\
\midrule
\multirow{3}{*}{\textbf{d}} 
& $\mathbf{50}$ & $70.06$ & $167.79$ & $236.34$ \\
& $\mathbf{100}$  & $61.11$ & $109.26$ & $160.44$ \\
& $\mathbf{150}$ & $54.60$ & $103.88$ & $134.99$ \\
\bottomrule
\end{tabular}
\end{table}

\section{Conclusions}
\label{sec:conclusion}
In this work, we presented a structured and modular framework for the reduction and control of large-scale agent-based models (ABMs), with the specific goal of steering the collective system dynamics toward consensus. The proposed methodology tackles two major sources of complexity inherent in ABMs: the large number of interacting agents and the high dimensionality of their individual state spaces. To address these challenges, we designed a two-stage reduction strategy that first reduces the number of agents through clustering and then compresses the state dimension using projection-based model order reduction techniques.

In the first stage, we apply clustering techniques to group agents exhibiting similar behaviour. This step produces a coarser system in which each cluster is represented by a center of mass, effectively capturing the dominant macroscopic dynamics. In the second stage, we perform a dimensionality reduction using Proper Orthogonal Decomposition, which identifies a low-dimensional subspace that best approximates the trajectory data of each cluster. This results in a set of reduced-order ODEs, one for each cluster, that retain the key structural properties of the original agent-based model.

These reduced dynamics serve as the basis for an optimal control problem formulated using Pontryagin’s Maximum Principle. The resulting control inputs are computed on the simplified model and subsequently lifted back to the full agent space through a suitable reconstruction procedure. This hybrid strategy enables the design of control policies that are both computationally efficient and dynamically meaningful.

The numerical experiments confirm that the direct application of optimal control strategies to full-scale ABMs becomes increasingly infeasible as the system scales in size and complexity. In such settings, control inputs either become too costly to compute or lose effectiveness due to the high-dimensional nature of the dynamics. Our reduction-and-control framework overcomes these obstacles by preserving the essential dynamical features necessary for consensus while drastically reducing the computational workload. As a result, it enables the deployment of feedback-based coordination mechanisms even in scenarios where direct methods would be prohibitive.

Beyond the first-order framework considered in this work, the proposed reduction-and-control methodology can be naturally extended to second-order agent-based models, such as the well-known Cucker–Smale dynamics \cite{choi2017emergent}. These models incorporate velocity as an additional state variable and are widely used to describe flocking and alignment behaviours in both natural and artificial systems. The second-order structure introduces further complexity due to the enlarged state space and the coupled nature of position and velocity dynamics. Nevertheless, the same principles of clustering, projection-based model order reduction, and optimal control design remain applicable. 
Moreover, this framework can also be adapted to agent-based models exhibiting more complex behaviors beyond simple alignment, such as intermittent behaviour \cite{Monti_onoff,Diele_onoff} and spatial pattern formation \cite{alla2023adaptive,ALLApdmd,BOZZINI2021}. 
These extensions demonstrate not only the technical feasibility of the framework, but also its practical relevance across a broader class of applications.

The proposed control and reduction framework can be extended to more complex agent-based models, including those arising in the study of opinion dynamics and collective risk perception. A particularly relevant example is provided by the O.R.E. model developed in \cite{giardini2021opinion}, which investigates how institutional communication and social influence jointly shape public perception in the face of uncertain, potentially catastrophic events such as natural disasters or pandemics.

In such contexts, agents process information received from both centralized (institutional) and decentralized (peer-to-peer) sources, with opinion formation influenced by individual attitudes such as trust and risk sensitivity. As discussed in that work, the dynamics often lead to persistent fragmentation or polarization - especially when institutional credibility is low or social amplification is strong. While the original model emphasizes the descriptive understanding of these dynamics, our framework offers a complementary perspective: it focuses on designing optimal interventions to enforce consensus as a purely dynamical coordination problem.

Although our control is agnostic to opinion content and does not aim to promote a specific belief, its use could be instrumental in analyzing whether and how timely, targeted interventions (e.g., public messaging or incentives) could synchronize opinion clusters or accelerate convergence in situations where consensus is desirable - such as building collective awareness of environmental risks. This makes our method potentially valuable not only as a computational benchmark but also as a basis for testing the structural controllability of real-world communication processes under risk and uncertainty.

\section*{Acknowledgements}

\noindent A.M. and F.D. research activity is funded by the National Recovery and Resilience Plan (NRRP), Mission 4 Component 2 Investment 1.4 - Call for tender No.
3138 of 16 December 2021, rectified by Decree n.3175 of 18 December 2021 of Italian Ministry of University and Research funded by the European Union – NextGenerationEU; Award Number: Project code CN 00000033, Concession Decree No. 1034 of 17 June 2022 adopted by the Italian Ministry of University and Research, CUP B83C22002930006, Project title \lq \lq National Biodiversity Future Centre''. A.M. acknowledges support from a scholarship (\emph{borsa per l’estero}) granted by Istituto Nazionale di Alta Matematica (INdAM) to carry out a research stay at Imperial College London. A.M. and F.D. are members of the INdAM research group GNCS. F.D. and A.M. would like to thank Mr. Cosimo Grippa for his valuable technical support.

\bibliographystyle{plain} 
\bibliography{biblio.bib}

\end{document}